\documentclass[11pt]{amsart}
\usepackage{geometry} 
\usepackage{pgfplots}
\usepackage{psfrag}
\usepackage{amsmath}
\usepackage{calc}
\usepackage{systeme}
\usepackage{amsfonts}
\usepackage{amsthm}
\usepackage{algorithm2e}
\usepackage[applemac]{inputenc}
\usepackage{bbold}
\usepackage[numbers]{natbib}

\numberwithin{equation}{section}

\author{%
  Charles Bertucci $^1$  }

\newtheorem{Theorem}{Theorem}[section]

\newtheorem{Rem}[Theorem]{Remark}
\newtheorem{Def}[Theorem]{Definition}
\newtheorem{Prop}[Theorem]{Proposition}

\newtheorem{Ex}[Theorem]{Example}

\usepackage[foot]{amsaddr}

\newcommand{\be}{\begin{equation}}
\newcommand{\ee}{\end{equation}}
\newcommand{\ba}{\begin{aligned}}
\newcommand{\ea}{\end{aligned}}

\newcommand{\R}{\mathbb{R}}
\newcommand{\T}{\mathbb{T}}

\newcommand{\mptd}{\mathcal{P}(\mathbb{T}^d)}
\newcommand{\mpptd}{\mathcal{P}(\mathcal{P}(\mathbb{T}^d))}

\title{Mean Field Games with incomplete information}
\thanks{$^1$ : CEREMADE, CNRS, Universit\'e Paris Dauphine-PSL, UMR 7534, 75016 Paris, France
}
\date{} 

\begin{document}
\maketitle
\begin{abstract}
This paper is concerned with mean field games in which the players do not know the distribution of the other players. First a case in which the players do not gain information is studied. Results of existence and uniqueness are proved and discussed. Then, a case in which the players observe the payments is investigated. A master equation is derived and partial results of uniqueness are given for this more involved case.
\end{abstract}
\tableofcontents

\section*{Introduction}
This paper is interested in Mean Field Games (MFGs) in which the players do not have a complete information on the distribution of the other players in the state space. Namely they are mostly unable to observe each other and have only an a priori on the initial distribution of players. This type of MFG leads to new mathematical questions which are partially solved here.

MFGs have attracted quite a lot of attention since the seminal work \citep{lasry2007mean,lions2007cours}. They are differential games involving non-atomic agents. MFG arise in a wide variety of modeling context such as economics \citep{krusell1998income,achdou2017income}, financial engineering \citep{carmona2020applications}, epidemiology \citep{hubert2018nash} or telecommunications \citep{bertucci2018transmit}. For a vast majority of the existing literature, it is always assumed that all the players have a complete information on the MFG, i.e. they can observe at any time the state and action of each player. In this paper, cases in which all the information is not available to the players are studied.

Several authors have studied problems in which the players do not know directly their individual state but only have some partial information on it, see for instance \citep{sen2016mean,sen2019mean,firoozi2020epsilon} for detailed studies of such cases. P.-L. Lions studied a MFG in which all the players are learning an unknown parameter of the model. The paper \citep{shmaya} also addresses the question of learning/playing at the same time. These setups are different from the one we study here. The closest work to ours in terms of models is \citep{casgrain2018mean} in which the authors studied a MFG in which the players do not know the controls of the other players, only the effect they have (as a whole) on their objective function. This last work relies mostly on the fact that their model is semi-explicitly solvable. We can also mention the work \citep{bergault} which studies questions of information in a major-minor type setting.
Independently of MFG, transport problems on the set of probability measures over probability measures have attracted attention recently, see for instance \citep{korba,pinzi,pinzi2}.

The rest of the paper is organized as follows. A presentation of the MFG model and a quick discussion on the structure of information in MFG is first. The rest of the paper is divided in two parts which constitute the core of the paper. The first one is concerned with a case in which the players have an incomplete initial information and do not gain any information with time. The second one is devoted to the situation in which the players do not observe the state of the other players but have complete information on all the payments.

\section{The MFG model and the classical structure of information}
\subsection{Presentation of the model}
We present here the framework of the underlying game between the players. The state of each players is a process valued on the $d$ dimensional torus $\mathbb{T}^d$ which evolves according to 
\begin{equation}\label{sde}
dX_t = \alpha_t dt + \sqrt{2\sigma}dW_t,
\end{equation}
where $(W_t)_{t \geq 0}$ is a $d$ dimensional Brownian motion on a standard (fixed) filtered probability space $(\Omega, \mathbb{P}, (\mathcal{F}_t)_{t \geq 0})$. The game lasts a time $T > 0$ and the cost of a player who uses the control $(\alpha_s)_{s \geq 0}$ is given by 
$$
\int_0^Tf(m_s)(X_s) + L(X_s,\alpha_s)ds + U_0(m_T)(X_T),
$$
where $(m_t)_{t \geq 0}$ is the evolution of the measure describing the spatial distribution of players and where $L, f$ and $U_0$ are cost functions on which assumptions are made later on. Players are allowed to choose adapted controls with respect to the $\sigma$-algebra generated by their state process. Clearly the cost paid by the players is unknown to them at the initial time since the evolution of their state is stochastic. We naturally assume that the players are risk neutral and take into account the expected cost they are too face which is, if the evolution $(m_t)_{t \in [0,T]}$ is known,
$$
\mathbb{E}_{\mathbb{P}}\left[\int_0^Tf(m_s)(X_s) + L(X_s,\alpha_s)ds + U_0( m_T)(X_T)\right].
$$
We do not particularly insist on why we make such an assumption, which is wildly common in the literature on stochastic optimal control. Hence, given an anticipation $(m_t)_{t \in [0,T]}$, a player can compute its optimal response by solving the Hamilton-Jacobi-Bellman (HJB) equation
\be\label{hjb1}
\begin{split}
-\partial_t u(t,x) - \sigma \Delta u(t,x) + H(x,\nabla_x u(t,x)) = f(m_t)(x) \text{ in } (0,T)\times \mathbb{T}^d,\\
u(T,x) = U_0(m_T)(x) \text{ in } \mathbb{T}^d,
\end{split}
\ee
where we have introduced the Hamiltonian $H(x,p) := \sup_{\alpha}\{-\alpha\cdot p - L(x,\alpha)\}$. The associated optimal control is given in feedback form by 
$$
\alpha_t = -D_pH(X_t,\nabla_x u(t,X_t)).
$$
On the other hand, given that the players use a strategy of the form $\alpha_t = b(t,X_t)$ for some function $b:[0,T]\times \T^d \to \R^d$, their distribution in the state space evolves according to the Fokker-Planck equation
\be\label{fp1}
\partial_t m_t - \sigma \Delta m_t + \text{div}(bm_t) = 0 \text{ in } (0,T)\times \mathbb{T}^d,
\ee
which is, as usual, understood in the sense of distribution. Hence, given an initial distribution of players $m_0 \in \mptd$, a strategic equilibrium is reached if one can find a solution $(u,m,b)$ of \eqref{hjb1} and \eqref{fp1} together with $b(t,x) = -D_pH(x,\nabla_xu(t,x))$. This is summarized in the system
\begin{equation}\label{mfg}
\begin{aligned}
&-\partial_t u - \nu \Delta u + H(x,\nabla u) = f(m)(x) \text{ in } (0,T)\times \T^d,\\
&\partial_t m - \nu \Delta m - \text{div}\left(D_pH(x,\nabla u)m\right) = 0 \text{ in } (0,T)\times \T^d,\\
&u(T,x) = G(m_T)(x), m|_{t = 0} = m_0 \text{ in } \T^d ,
\end{aligned}
\end{equation}
where the dependence of the unknown $(u,m)$ in $(t,x)$ is omitted to lighten the notation.

\subsection{The structure of information}
In the previous system, if neither the particular form of the second order term or the fact that dependence in $m$ and $\nabla u$ are decoupled are important, a fundamental observation lies in the initial distribution of players $m_0$. This observation is that the knowledge of $m_0$ is equivalent (in terms of induced equilibria) to the knowledge of the whole evolution of the distribution of players $(m_t)_{t \in [0,T]}$. This quite simple fact does not need any particular proof as it only suffices to remark that $m_0$ is the only datum in the previous system of equations. Of course, this is a consequence of the deterministic evolution of $(m_t)_{t \in [0,T]}$, given the strategies of the players. In other words, even if the players do not observe each other during the game, as long as they know $m_0$, the induced equilibria are the same as if they observe the whole trajectory $(m_t)_{t \in [0,T]}$. Because the structure (or set) of equilibria only depends on the initial distribution $m_0$, the following question seems natural : what happens to the structure of equilibria if the players do not know $m_0$ ?

We insist that in several models, it is quite natural that players do not know $m_0$. Indeed, for instance, in the telecommunication model of \citep{bertucci2018transmit}, the state of a player (or device) is the amount of data this device wants to transmit, which is private. The same goes for financial models such as the one of \citep{cardaliaguet2016mean}, where the state of a player (or trader) is her portfolio.

To address the question of incomplete knowledge of $m_0$, some assumptions have to be made on the knowledge at each instant that the players have on the distribution of other players. However, let us state that if the process $(m_t)_{t \geq 0}$ is not known and the players have a prior $(µ_t)_{t \geq 0}$ on it, a risk neutrality assumption shall be made on the players. By prior we mean that instead of anticipating an evolution $(m_t)_{t \geq 0}$ for the distribution of players, the players believe that at any time $t \geq 0$, the distribution of players is unknown and that this uncertainty is described by the measure $µ_t \in \mathcal{P}(\mathcal{P}(\mathbb{T}^d))$. In such a context, the new expected cost of the players is
\begin{equation}\label{costmu}
\mathbb{E}_{\mathbb{P}}\left[\int_0^T\int_{\mptd}f(m)(X_s)µ_s(dm) + L(X_s,\alpha_s)ds + \int_{\mptd}U_0( m)(X_T)µ_T(dm)\right].
\end{equation}

\subsection{Assumptions and notation}
We now present the standing assumptions for the rest of the paper. Before that, let us recall some properties of sets of probability measures. 

Assume that $(E,d)$ is a compact and Polish (i.e. complete separable metric) space. We denote by $\mathcal{M}(E)$ the set of Borel measures on $E$ and by $\mathcal{P}(E)$ the set of Borel probability measures on $E$. The latter can be equipped with the distance $\textbf{d}_1$ defined by
$$
\textbf{d}_1(µ,\nu) := \sup_{\phi} \int_E \phi(x)(µ-\nu)(dx),
$$
where the supremum is taken over Lipschitz functions on $(E,d)$ with a Lipschitz constant of at most $1$. The set $(\mathcal{P}(E),\textbf{d}_1)$ is compact and Polish. In all this paper, $\mathcal{P}(E)$ is always seen as equipped with $\textbf{d}_1$. In particular, $\mpptd$ is a compact set. For a function $\phi : \mptd \to \mathbb{R}$, we note for $m \in \mptd, x \in \mathbb{T}^d$
$$
\nabla_m\phi (m,x) = \lim_{\theta \to 0}\frac{\phi((1-\theta)m + \theta \delta_x)}{\theta},
$$
when it is defined.

The image of a measure $µ \in \mptd$ by a map $T : \T^d \to \T^d$ is denoted by $T_{\#}µ$.

For the rest of the paper, we assume the following
\begin{itemize}
\item The hamiltonian $H$ is smooth(, convex) and globally Lipschitz continuous in $p$, uniformly in $x$.
\item The function $f$ (respectively $U_0$) is continuous from $\mptd$ to $\mathcal{C}^{\alpha}(\mathbb{T}^d)$ (respectively to $\mathcal{C}^{2,\alpha}(\mathbb{T}^d)$) for some $\alpha > 0$.
\end{itemize}

Let us also recall that given a duality product $\langle\cdot,\cdot \rangle$ between two sets $E$ and $E'$, a mapping $F : E' \to E$ is said to be 
\begin{itemize}
\item monotone if for all $x,y \in E'$
$$
\langle F(x)-F(y),x -y\rangle \geq 0.
$$
\item strictly monotone if for all $x,y \in E'$
$$
\langle F(x)-F(y),x -y\rangle = 0 \Rightarrow F(x) = F(y)
$$
\end{itemize}

\section{The blind case}
The model we study in this section is going to be called the blind case. In this situation the players all start with a common belief $µ_0 \in \mathcal{P}(\mathcal{P}(\mathbb{T}^d))$ on the initial distribution of players and they do not gain any information during the game. By this we mean that they only observe their individual state for the whole duration of the game. This common belief can be thought of as a public information. If the players have the anticipation $(\tilde{µ}_t)_{t \geq 0}$ for their belief, from \eqref{costmu}, the optimization problem they have to face is described by the HJB equation
\be\label{HJB}
\begin{aligned}
-\partial_t u - \sigma \Delta u + H(x,\nabla u) &= \int_{\mptd}f(m)(x)\tilde{µ}_t(dm) \text{ in } (0,T) \times \mathbb{T}^d ;\\
u(T,x) &= \int_{\mptd}U_0(m)(x)\tilde{µ}_T(dm) \text{ in } \mathbb{T}^d,
\end{aligned}
\ee
from which they can compute an optimal response given by $\alpha_t = -D_pH(X_t,\nabla_xu(t,X_t))$. Let us insist on the fact that, in this case, the players do not observe the cost they are paying, only their state. Indeed, in this model, the players start with an evolution of an a priori on the distribution of players, and they stick to this a priori during the game. If they were to observe their payments, they would have to take into account this information in their belief, which is described in the next section.

\begin{Rem}
In this model, the assumption that the players are non-atomic is essential. Indeed, we stipulate that they only observe their state, and since they are non-atomic players, this information does not carry any information about the value of the measure $m$ describing the distribution of players. Of course this would not be the case in a standard $N$ player game.
\end{Rem}

\subsection{Evolution of the belief of the players}
It remains to describe the evolution of the belief of the players. Under the anticipation that the strategies of the players is going to be given by $b : [0,T]\times \mathbb{T}^d \to \mathbb{R}^d$, if the initial distribution of players is $m_0$, recall that its evolution can be computed through the Fokker-Planck equation
\be\label{FP}
\partial_t m - \sigma\Delta m + \text{div}(bm) = 0 \text{ in } (0,\infty)\times \mptd,
\ee
with the initial condition
$$
m|_{t = 0} = m_0.
$$
The previous is true whatever the initial condition since $b$ does not depend on $m$. Hence the evolution of the belief of the players is the push forward of the initial belief by this Fokker-Planck equation. To be more precise, denote by $K_t(m_0)$ the solution of \eqref{FP} at time $t$ with initial condition $m_0$. This defines a semi group of operators $(K_t)_{t \geq 0}$. Given an initial belief $µ_0$ (and anticipations $b$), the belief $µ_t$ on the distribution of players at time $t$ is given by
\be\label{defpush}
µ_t = (K_t)_{\#}µ_0.
\ee
Formally, we can characterize the evolution $(µ_t)_{t \geq 0}$ given by \eqref{defpush} with the following continuity equation
\be\label{contpp}
\begin{aligned}
\partial_t µ &+ \nabla_m \cdot\bigg((\sigma\Delta m - \text{div}(mb))µ\bigg) = 0 \text{ in } (0,T)\times \mptd,\\
&µ|_{t = 0} = µ_0.
\end{aligned}
\ee
The operator $\nabla_m\cdot$ is thought of as a divergence operator on $\mptd$ and it can be understood in a dual manner as the following result explains. This continuity equation on $\mptd$ states that the weight that $µ_t$ puts on any element of $\mptd$ is transported along the paths generated by the Fokker-Planck equation \eqref{FP}.

\begin{Prop}\label{weaksol}
Fix $µ_0$ and let $(µ_t)_{t \geq 0}$ be defined by \eqref{defpush} (for a given smooth function $b$). For any smooth\footnote{By smooth we mean that both $\phi$ and $\nabla_m\phi$ are well defined and smooth functions of $(t,x)$.} function $\phi : [0,T]\times \mptd \to \mathbb{R}$ such that $\phi(T) = 0$, the following holds
$$
\begin{aligned}
\int_0^T \int_{\mptd}&\left(-\partial_t \phi(t,m) - \int_{\mathbb{T}^d}(\sigma\Delta_x \nabla_m \phi(t,m,x) + b(t,x)\cdot \nabla_x \nabla_m \phi(t,m,x))m(dx) \right)µ_t(dm) dt\\
& - \int_{\mptd}\phi(0,m)µ_0(dm) = 0.
\end{aligned}
$$
Moreover, it is the unique process to satisfy the previous variational relation.
\end{Prop}
\begin{proof}
It suffices to compute for any $t \geq 0$
\be\label{eqproof2}
\begin{aligned}
\int_{\mptd}\phi(t,m)µ_t(dm) &= \int_{\mptd}\phi(t,m)(K_t)_{\#}µ_0(dm)\\
& = \int_{\mptd}\phi(t,K_tm)µ_0(dm)
\end{aligned}
\ee
and to remark that for any $t \geq 0, m \in \mptd$
\be\label{eqproof1}
\ba
\frac{d}{dt} \phi(t,K_tm) &= \partial_t \phi(t,K_tm) + \int_{\mathbb{T}^d}\left(\sigma \Delta K_t m - \text{div}(bK_tm)\right)\nabla_m\phi(t,K_tm,x)dx\\
&=  \partial_t \phi(t,K_tm) + \int_{\mathbb{T}^d}\left(\sigma \Delta \nabla_m\phi(t,K_tm,x) + b(t,x)\cdot\nabla_x \nabla_m\phi(t,K_tm,x)\right)(K_tm)(dx)
\ea
\ee
Because $\phi$ is smooth, namely its derivative with respect to $m$ is smooth in $x$, the previous integral is well defined. Integrating \eqref{eqproof1} between $0$ and $T$ and using \eqref{eqproof2} gives the first part of the result.

The second part is obtained by taking two such processes and by considering their difference $\nu$. By construction, this difference satisfies for any $\phi$ as in the statement 
\be\label{eqnu1}
0 =\int_0^T \int_{\mptd}\left(-\partial_t \phi(t,m) - \int_{\mathbb{T}^d}(\sigma\Delta_x\nabla_m\phi(t,m,x) + \nabla_x\nabla_m\phi(t,m,x)\cdot b(t,x)) m(dx) \right)\nu_t(dm) dt.
\ee
Take any smooth function $G : [0,T]\times \mptd \to \mathbb{R}$ and define 
$$
\phi(t,m) := \int_t^TG(s,K_{s-t}m)ds.
$$
Now observe that for all $m \in \mptd$, $\phi(T,m) = 0$. Moreover, since $G$ is smooth, differentiating the previous equation with respect to $t$ yields that for all $t\in(0,T), m \in \mptd$,
$$
-\partial_t \phi(t,m) - \int_{\mathbb{T}^d}(\sigma\Delta_x\nabla_m\phi(t,m,x) + \nabla_x\nabla_m\phi(t,m,x)\cdot b(t,x)) m(dx) = G(t,m).
$$
Hence, plugging $\phi$ in \eqref{eqnu1} yields
$$
\int_0^T \int_{\mptd}G(t,m)\nu_t(dm)dt = 0.
$$
Hence $\nu = 0$ by density of smooth functions in $\mathcal{C}(\mptd)$. We refer to \citep{lions2007cours,cardaliaguet2010notes} for the density of smooth functions in $\mptd$.
\end{proof}
\begin{Rem}
The question of existence of such a path $(µ_t)_{t \geq 0}$ has already been answered since it has been constructed above.
\end{Rem}
Moreover, we can establish an estimate on the evolution of the belief with respect to time, given that $b \in L^{\infty}$.
\begin{Prop}\label{prop:cont}
Assume that the drift $b$ of \eqref{FP} is bounded. Then $(µ_t)_{t \in [0,T]}$ defined in \eqref{defpush} is uniformly $\frac12$-H\"older continuous, with a constant depending only on $\|b\|_{\infty}$ and $T$.
\end{Prop}
\begin{proof}
By definition, for $s,t \in [0,T]$
$$
\begin{aligned}
\textbf{d}_1(µ_s,µ_t) :&= \sup_{\|\phi\|_{Lip} \leq 1} \int_{\mptd}\phi d(µ_s-µ_t)\\
& =  \sup_{\|\phi\|_{Lip} \leq 1}\int_{\mptd} \phi(K_s m) - \phi(K_t m) µ_0(dm)\\
& \leq \int_{\mptd}\textbf{d}_1(K_s m,K_t m)µ_0(dm)\\
& \leq C_0 \sqrt{|t-s|},
\end{aligned}
$$
from classical $\frac12$-H\"older continuity estimate on the Fokker-Planck equation, see for instance \citep{cardaliaguet2010notes}.
\end{proof}

\begin{Ex}
To illustrate the previous evolution of the belief, consider the case in which $µ_0$ is a combination of Dirac masses. If it is given by
$
µ_0 := n^{-1}\sum_{i =1}^n \delta_{m_i},
$
then for any $t\geq 0$, $µ_t$ is simply given by
$
µ_t := n^{-1}\sum_{i=1}^n \delta_{K_tm_i}.
$
\end{Ex}

\subsection{Existence of Nash equilibria of the game}
In the same way as \eqref{mfg} characterizes Nash equilibria of the MFG when the initial distribution of players is known, we can characterize Nash equilibria of the MFG with incomplete information with a system of PDE. Indeed, given an initial belief $µ_0$, and a profile of strategy $b : [0,T] \times \mathbb{T}^d \to \mathbb{R}^d$ for the players, one can compute the associated evolution of the belief with \eqref{contpp}. Hence, a best response is derived through the HJB equation \eqref{HJB} whose solution is $u$. And in the end, one indeed gets a Nash equilibrium if $b (t,x) = -D_pH(x,\nabla_x u(t,x))$.

Thus Nash equilibria of the blind game with initial distribution $µ_0$ are characterized as solutions of
\be\label{mfgblind}
\begin{aligned}
-\partial_t u& - \sigma \Delta u + H(x,\nabla u) = \int_{\mptd}f(m)(x)µ_t(dm) \text{ in } [0,T] \times \mathbb{T}^d ;\\
\partial_t µ& + \nabla_m \cdot\bigg(\big(\sigma\Delta m + \text{div}(mD_pH(\cdot,\nabla_x u))\big)µ\bigg) = 0 \text{ in } (0,T)\times \mptd,\\
u(T,x) = &\int_{\mptd}U_0(m)(x)µ_T(dm) \text{ in } \mathbb{T}^d, \quadµ|_{t = 0} = µ_0.
\end{aligned}
\ee
To lighten the notation, we introduce $\tilde{f} : µ \in \mpptd \to (x \to \int_{\mptd}f(m)(x)µ(dm) )\in \mathcal{C}^{\alpha}(\mathbb{T}^d)$ and analogously $\tilde{U}_0 : µ \to (x \to \int_{\mptd}U_0(m)(x)µ(dm)) \in \mathcal{C}^{2,\alpha}(\T^d)$. Let us observe that those two functions are linear in $µ$.
\begin{Rem}
If we assumed that the agents are not risk neutral, but take into account their belief in a different way, then functions $\tilde{f}$ and $\tilde{U}_0$ which are not necessary linear could be obtained. For smooth functions this generalization is not difficult and the reader can check that the linearity of these functions does not play a role in the following result of existence.
\end{Rem}

We can establish the following result.
\begin{Theorem}
There exists a solution $(u,µ)$ of \eqref{mfgblind} in the sense that $u$ is a classical solution of the Hamilton-Jacobi-Bellman equation and $µ$ is the unique solution of the continuity equation in the sense of Proposition \ref{weaksol}.
\end{Theorem}
\begin{proof}
Let us consider the applications $\psi_1, \psi_2$ and $\psi_3$ defined by : $\psi_1 : \mathcal{C}([0,T],\mpptd) \to \mathcal{C}^{1,2}([0,T]\times\mathbb{T}^d)$ associates to  $(µ_t)_{t \geq 0}$ the solution of the HJB equation \eqref{HJB} ; $\psi_2 :  \mathcal{C}^{1,2}([0,T]\times\mathbb{T}^d) \to \mathcal{C}([0,T],\mpptd)$ associates to a function $u$ the solution of \eqref{contpp} with initial condition $µ_0$ and drift $b = -D_pH(\nabla_x u)$ ; and $\psi_3 := \psi_2\circ \psi_1$.

To prove that $\psi_3$ has a fixed point, using Schauder's fixed point Theorem, it is sufficient to establish that $\psi_2$ is a compact and continuous mapping as $\psi_1$ is clearly continuous, thanks to the stability of the HJB equation that arises from the comparison principle, which is true in much more general setting, see for instance \citep{crandall1992user}.\\

The fact that $\psi_2$ is compact is a direct application of Proposition \ref{prop:cont}. The fact that it is continuous is a consequence of the continuity of the Fokker-Planck \eqref{FP} equation with respect to the drift term $b$ and on the definition of $\psi_3$ using \eqref{defpush}.
\end{proof}
\begin{Rem}
Even though the minimal regularity of the previous system is not the main concern of this paper, one can immediately check that the only regularity needed on $u$ is $W^{1,2,\infty}$ and not the $\mathcal{C}^{1,2,\alpha}$ regularity which is here a consequence of the regularity of $f$ and $U_0$. Hence these assumptions on $f$ and $U_0$ can be easily weakened.
\end{Rem}

The previous result of existence of such equilibria is in no sense surprising and falls in a category of somehow classical result of existence of Nash equilibria in MFG. A more interesting question is the effects the lack of knowledge has on the uniqueness properties of the equilibria, which we now address.

\subsection{Uniqueness of Nash equilibria}
Let us first compute the usual proof of uniqueness of MFG which dates back to the original paper of Lasry and Lions \citep{lasry2007mean}. Take two solutions $(u_1,µ_1)$ and $(u_2,µ_2)$ of the system \eqref{mfgblind}. Take the difference of the left hand sides of the HJB equations, integrate against an arbitrary measure $m \in \mptd$ and then integrate once again against the difference $µ_1 - µ_2$, doing this, we obtain using the equations satisfied by $µ_1$ and $µ_2$
\be\label{computeblind}
\begin{aligned}
\int_0^T& \int_{\mptd}\int_{\mathbb{T}^d}((-\partial_t- \sigma \Delta)(u_1-u_2) + H(x,\nabla u_1) - H(x,\nabla u_2))m(dx)(µ_1(t) - µ_2(t))(dm)dt\\
=& \int_0^T\int_{\mptd}\int_{\mathbb{T}^d}\left(\nabla_x (u_1 - u_2)\cdot D_pH(x,\nabla u_1) + H(x,\nabla u_1)- H(x, \nabla u_2)\right)m(dx)µ_1(t)(dm)dt\\
&+\int_0^T\int_{\mptd}\int_{\mathbb{T}^d}\left(\nabla_x (u_2 - u_1)\cdot D_pH(x,\nabla u_2) + H(x,\nabla u_2)- H(x, \nabla u_1)\right)m(dx)µ_2(t)(dm)dt\\
& - \int_{\mptd}\int_{\mathbb{T}^d}\tilde{U}_0(µ_1(T)- µ_2(T))(x)m(dx)(µ_1(T) - µ_2(T))(dm).
\end{aligned}
\ee
Using the convexity of the Hamiltonian, and the HJB equations, we obtain
$$
\begin{aligned}
\int_{\mptd}&\int_{\mathbb{T}^d}\tilde{U}_0(µ_1(T)- µ_2(T))(x)m(dx)(µ_1(T) - µ_2(T))(dm)\\
& + \int_0^T \int_{\mptd}\int_{\mathbb{T}^d}\tilde{f}(µ_1(t)- µ_2(t))(x)m(dx)(µ_1(t) - µ_2(t))(dm)dt \leq 0.
\end{aligned}
$$
This somehow classical computation yields a uniqueness result which is analogous to the usual result of uniqueness for MFG Nash equilibria.
\begin{Theorem}\label{thm:uniqnash}
Assume that $f$ and $U_0$ are such that 
$$
\int_{\mptd}\int_{\mathbb{T}^d}\tilde{U}_0(µ_1 - µ_2)(x)m(dx)(µ_1 - µ_2)(dm)\leq 0 \Rightarrow \tilde{U_0}(µ_1) = \tilde{U_0}(µ_2), \forall µ_1,µ_2 \in \mpptd,
$$
\be\label{monbiz}
\int_{\mptd}\int_{\mathbb{T}^d}\tilde{f}(µ_1 - µ_2)(x)m(dx)(µ_1 - µ_2)(dm)\leq 0\Rightarrow \tilde{f}(µ_1) = \tilde{f}(µ_2), \forall µ_1,µ_2 \in \mpptd.
\ee
Then there is at most one solution $(u,µ)$ of \eqref{mfgblind}.
\end{Theorem}
The next result states that these conditions are stronger than classical monotonicity for $f$ and $U_0$, and that there exist functions satisfying them. 
\begin{Prop}
\begin{itemize}
\item Any function $f$ which satisfy the requirements of the previous theorem is a monotone operator, in the sense of the duality between continuous functions and measures.
\item For any smooth function $\phi(t,x)$, if $f$ is defined by 
$$
f(t,x,m) = \phi(t,x)\int_{\mathbb{T}^d}\phi(t,y)m(dy),
$$
then it satisfies the assumption of the previous result.
\end{itemize}
\end{Prop}
\begin{proof}
The first claim follows immediately from choosing $µ_1$ and $µ_2$ as Dirac masses.\\

The second one follows from the computation
$$
\begin{aligned}
\int_{\mptd}&\int_{\mathbb{T}^d}\int_{\mptd}f(x,m')(µ_1 - µ_2)(dm')m(dx)(µ_1 - µ_2)(dm)\\
&= \int_{\mptd}\int_{\mathbb{T}^d}\int_{\mptd} \phi(t,x) \int_{\mathbb{T}^d}\phi(t,y)m'(dy)(µ_1 - µ_2)(dm')m(dx)(µ_1 - µ_2)(dm)\\
&= \left(\int_{\mptd}\int_{\mathbb{T}^d}\phi(t,y)m'(dy)(µ_1 - µ_2)(dm')\right)^2.
\end{aligned}
$$
\end{proof}

If it clear that the previous example of existence of functions satisfying the requirements of Theorem \ref{thm:uniqnash} can be generalized, for instance by adding terms independent of $m$. It is also clear that this requirement is more restrictive than the monotonicity. Indeed, consider the following example which highlights the fact that the linearity of $f$ is very helpful to obtain the monotonicity of $\tilde{f}$.
\begin{Ex}
Assume that $d =1$ and that $f$ is given by 
$$
f(m)(x) = xg\left(\int_{\T}ym(dy)\right),
$$
for some continuous increasing function $g:[0,1] \to \R$. The function $f$ is monotone as for $m,m' \in \mptd$,
$$
\ba
\int_{\T}&f(m)(x) -f(m')(x) (m-m')(dx)\\
& = \left(g\left( \int_{\T} y m(dy)\right) - g\left(\int_{\T}ym'(dy)\right)\right)\left(\int_{\T}xm(dx) - \int_{\T}x m'(dx)\right)\\
& \geq 0,
\ea
$$
since $g$ is non decreasing. But, as soon as $g$ is not affine, $\tilde{f}$ does not verify \eqref{monbiz}. Indeed, if $g$ is not affine, then there exists $x,y \in \T$ such that $g(x) + g(y) \ne 2 g(\frac{x+y}{2})$. Without loss of generality we can assume that $g(x) + g(y) < 2 g(\frac{x+y}{2})$. Hence, we can consider $z \in (0,1)$ such that $z < \frac{x + y }{2}$ and  $g(x) + g(y) < 2 g(z)$. Compute now
$$
\ba
\int_{\mathcal{P}(\T)}&\int_{\mathbb{T}}\tilde{f}\left(\frac12 \delta_{\delta_x} + \frac12 \delta_{\delta_y} - \delta_{\delta_{z}}\right)(u)m(du)\left(\frac12 \delta_{\delta_x} + \frac12 \delta_{\delta_y} - \delta_{\delta_{z}}\right)(dm)\\
&=\left(\frac 12 (g(x) + g(y)) - g(z)\right)\left(\frac{x + y}{2} - z\right)\\
& < 0.
\ea
$$
\end{Ex}

Finally, let us insist on the fact that the systems \eqref{mfgblind} and \eqref{mfg} are very similar and that it seems unlikely that general uniqueness results can be obtained outside of the assumptions of Theorem \ref{thm:uniqnash} for \eqref{mfgblind}, in the same way as there are few uniqueness results for \eqref{mfg} outside of the usual monotone conditions. However, we do not claim to have answered fully this question and leave it open here.\\

Another quite simple, but rather important, information we can observe at the moment is that, if the indetermination on the initial distribution does not affect the payments, then the uniqueness argument on monotonicity is still valid. The next result is a more precise statement of this fact.

\begin{Prop}\label{rem:payments}
Let $f: \mptd \mapsto \mathcal C(\T^d)$, $\mathcal{A} := \{ µ \in \mpptd, \exists g, µ \text{ a.e. in } m, f(m) = g\}$. Then $\tilde f$ is monotone on $\mathcal A$.
\end{Prop}
\begin{proof}
Let $µ,\nu \in \mathcal A$, and take $m_1, m_2$ in respectively the supports of $µ$ and $\nu$. By monotonicity of $f$
\begin{equation}\label{eq:proofrem}
\int_{\T^d}f(m_1)(x)- f(m_2)(x)(m_1-m_2)(dx) \geq 0,
\end{equation}
and furthermore, by definition of $\mathcal A$, $\tilde f(µ) = f(m_1)$ and $\tilde f (\nu) = f(m_2)$. Hence, integrating \eqref{eq:proofrem} against $µ(dm_1)\otimes \nu(dm_2)$ yields the result after simple computations.
\end{proof}

\section{The case of observed payments}
The last remark of the previous section suggests to be interested in the case in which, at any time, the absence of knowledge on the distribution of players does not translate into an absence of knowledge on the payments of the players. Indeed if at any time the players know the cost $f$, even if they do not know exactly the distribution of players $m$, the cost $f$ "stays" monotone in some sense.

This section considers a case in which the players observe all the payments. By this we mean that at a time $t$, even if the players do not know exactly the actual distribution of players $m_t$, they know the cost $f(m_t) \in \mathcal{C}(\mathbb{T}^d)$ that it induces (on the whole state space). In this situation, the information of the costs is common to all the players and it can be thought of as a public information. This structure is very reminiscent of a common noise in MFG. That is why we are not going to try to characterize Nash equilibria of the game, but focus on trying to define a notion of value of the MFG, using the master equation.

This section is far from being a complete study of such models and its aim is more to introduce this problem. Some partial results are given. The rest of this section is organized as follows. After a formal description of the model and some reminders on the disintegration of measures, the evolution of the belief with information on the payments is presented. We then derive the associated master equation and present various partial results of uniqueness.\\

It is worth mentioning that, in a situation in which the cost function is injective, such a model is of no interest as the players learn instantly the distribution of players. However we argue that in several models (especially macro-economic ones such as in \citep{krusell1998income}), the cost function is far from being one to one and, for instance, depends only on a few moments of the distribution of players.

\subsection{The model}
The framework is the following. Because the players are going to update their belief on the distribution of players using the information they have on the payments, the belief $(µ_t)_{t\geq 0}$ can no longer be computed as a function of only the strategies of the players and the initial belief $µ_0$. Indeed, it will depend on the actual, unknown distribution of players through the information received through the payments. Even if $(µ_t)_{t \geq 0}$ is by no mean random here, it is convenient to use the standard probabilistic framework to understand the object at interest.\\

Even if the players do not know the initial distribution of players, there is an actual $m_0 \in \mptd$ which describes their initial distribution. Because the players have initially the belief $µ_0$, let us model $m_0$ as a $\mptd$ valued random variable whose law is $µ_0$. In particular we shall assume here that the belief is consistent with the actual distribution of players, i.e. that formally, the actual $m_0$ is in the support of $µ_0$. The evolution of the actual distribution of players is denoted $(m_t)_{t \geq 0}$. Because it is unknown to the players, it can also be modeled as a random process whose law is $(µ_t)_{t\geq 0}$. The players then observe at any time $t\geq 0$ the payments $f(m_t) : \mathbb{T^d} \to \mathbb{R}$ and update their belief accordingly by conditioning it on their observation. The next section explains how the belief is updated.\\

Moreover, it seems clear that the information process $(f(m_t))_{t \geq 0}$ plays the role of a common noise and that a deterministic approach using a forward-backward system such as \eqref{mfgblind} can no longer be sufficient to model equilibria of the MFG. Thus master equation approach to characterize a value is presented later on. We refer to \citep{cardaliaguet2019master} for more details on MFG master equations. A similar approach to the one of \citep{carmona2016common} (used to deal with common noise) seems to be usable. However we were not able to adapt these arguments to non smooth conditionings.

\subsection{Reminders on disintegration}
Since players are going to update their belief according to the new information they gain, some facts on the disintegration, or conditioning, of measures are recalled. If one were to compute an expectation of some random variable, given an a priori information, the proper object to use would be the conditional expectation. Even if, because players are risk neutral, expectations shall be taken, we believe the following attempt at describing the evolution of the belief is helpful to understand the situation. The process of obtaining the desired conditioning is called in the literature the disintegration of a measure.

Given $µ\in \mpptd$, a measurable set $E$ and a measurable function $\psi : \mptd \to E$, the disintegration of $µ$ along $\psi$ is a family $(µ_y)_{y \in F}$ of probability measures on $\mptd$, where $F =\psi(\mptd)$, such that 
\be\label{star}
\text{For any measurable set }A \subset \mptd, \quad µ(A) = \int_{\mptd}µ_{\psi(m)}(A)µ(dm).
\ee
\be\label{star2}
\text{For any }y\in F, \quad µ_y(\psi^{-1}(\{y\})) = 1.
\ee
Formally, $µ_y$ is the conditioning of $µ$ on the fact that the information $y$ has been received. Disintegrations could have been defined in more general settings, see section 452 of \citep{fremlin2000measure}. Their existence and uniqueness is an involved question. In this setting, the existence of a disintegration, and its uniqueness $\psi_{\#}µ$ almost everywhere hold, see III-70 in \citep{dellacherie-meyer}.

\subsection{Evolution of the belief}
This section describes the evolution of the belief of the players in the context of observed payments, given that the strategies of the players are given by a function $b : [0,T]\times \mathbb{T}^d \to \mathbb{R}^d$. First, recall that, because at any time the players observe the payments, the process $(µ_t)_{t \geq 0}$, which represents the common belief on the distribution of players, has to be valued in 
\be\label{defA}
\mathcal{A} := \{ µ \in \mpptd, \exists g, µ \text{ a.e. in } m, f(m) = g\}.
\ee 
Consider the total information the players have received up to the time $t\geq 0$, when the initial distribution of players is $m\in \mptd$. Denoting this information $\mathcal{F}(t,m)$, one finds
$$
\mathcal{F}(t,m) := (f(K_sm))_{s \in [0,t]}.
$$

The belief of the players evolves as a combination of the two "rules"
\begin{enumerate}
\item The process $(µ_t)_{t\geq 0}$ is weighting elements of $\mptd$ which are transported along the \textbf{same} Fokker-Planck equation. Indeed, the strategies of the players, hence the drift $b$ in \eqref{FP}, cannot depend on the different elements of $\mptd$ which are "weighted" by $µ$.
\item At any time, the belief $µ_t$ is disintegrated along the function $\mathcal{F}$ into $(µ_{\theta})_{\theta \in \mathcal{F}(\mptd)}$ and the belief $µ_{\theta'}$ which corresponds to the observed payments $\theta'$ becomes the new belief.
\end{enumerate}
Obviously the previous is quite formal and a more precise definition is presented below.  
 \begin{Prop}\label{proclearn}
 Given $\bar{µ} \in \mathcal{A}$, for $\bar{µ}$ almost every $m$, there exists a process $(µ^{m}_t)_{t \geq 0}$ which satisfies for any $t \geq 0$ 
 \begin{itemize}
 \item $µ^{m}_0 = \bar{µ}$.
 \item For any $t \geq 0$, there exists $\nu_t$ such that $µ^{m}_t = (K_t)_{\#}\nu_t.$ 
 \item For any $t \geq 0$, for $µ^{m}_t$ almost every $m'$, $f(m') = f(K_tm)$.
 \item For any measurables $X,A \subset \mptd$ such that for $\bar{µ}$ almost every $m' \in A$, $f(K_tm') = f(K_tm)$, $µ^{m}_t(K_t(A)\cap X) = µ^{m}_t(X)$.
\item For any $t\geq 0$, $\phi \in \mathcal{C}(\mptd)$, 
\be\label{star3}
\int_{\mptd}\int_{\mptd}\phi(m')µ^m_t(dm')\bar{µ}(dm) = \int_{\mptd}\phi(m')(K_t)_{\#}\bar{µ}(dm').
\ee
 \end{itemize}  \end{Prop}
 \begin{Rem}
 This proposition is almost the definition of the evolution of the belief in this model. The first point states the initial condition, the second one that this evolution follows the first rule above, the third one states that the belief is indeed consistent with the information and the fourth point states that the process $(µ^{\mathcal{F}(m)}_t)_{t\geq 0}$ is not too restrictive (not simply $(\delta_{K_tm})_{t \geq 0}$ for instance), even though no uniqueness result is stated here.
 \end{Rem}
\begin{proof}
For $t\geq 0$ consider the disintegration of $µ$ along $\psi := \mathcal{F}(t,\cdot)$ and denote it by $(\tilde{\nu}_{\theta})_{\theta \in \psi(\mptd)}$. Defining $\nu_t := \tilde{\nu}_{\mathcal{F}(t,m)}$ and \be µ^{\mathcal{F}(m)}_t := (K_t)_{\#}\nu_t \ee proves the claim. Indeed, the first three points are immediate. The fourth one follows from \eqref{star2} and the last one comes from \eqref{star}. Note that the disintegration is uniquely defined $\psi_{\#}\bar{µ}$ almost everywhere from III-70 in \citep{dellacherie-meyer}.
\end{proof}

\subsection{An illustrative example}
The previous approach was only presented to formally derive the master equation associated to this problem, and only few results are available in this situation. Nonetheless, we now present an example in which we can describe the evolution of the belief more explicitly. 

Consider the case $d = 1$ and assume that the initial belief $µ_0$ is given by $p_1 \delta_{\delta_0} + p_2 \delta_{\delta_{\epsilon}}$ where $\epsilon \in (0,\frac 14)$ and $p_1,p_2 \geq 0, p_1 + p_2 = 1$. Consider now a function $f_0$ such that
$$
f_0(x) = \begin{cases} 0 \text{ if } x \notin (\frac 14,\frac 12 - \frac{1}{16})\\ -2 \text{ if } x \in (\frac14 + \frac{1}{16}, \frac 12 - \frac 18), \end{cases}
$$
and $f_0$ is decreasing in $[\frac 14,\frac{5}{16}]$ and increasing in $[\frac 38,  \frac{7}{16}]$. We assume moreover that $f_0$ is globally smooth. We define the cost $f$ with,
$$
\forall m \in \mathcal{P}(\T), x \in \T, f(m)(x) = f_0(x) + cf_0(x)\int_{\T}f_0(y)m(dy),
$$
for $c \in (0,1)$. Finally, we assume that the duration of the game is $T=2$, that there is no final cost and that the Hamiltonian is given by $H(x,p) = |p|$, which corresponds to the situation in which the players can choose any control they like in $[-1,1]$ for a cost of $0$. 

In this situation, no matter what $m$ actually is, the players always prefer to be in the interval $[\frac{5}{16},\frac{3}{8}]$, hence, for any time $t \in [0,T]$, any player in $x \in [0,\frac{5}{16}]$ will always choose a control equal to $+1$. Hence, the belief at time $t \leq \frac{5}{16}- \epsilon$ will be given by 
$$
µ_t = p_1\delta_{\delta_t} + p_2\delta_{\delta_{\epsilon + t}}.
$$
Moreover, for these times, the players have only received the information process constant equal to $f_0$. Hence there is no information to take into account. Then, at time $t =( \frac{5}{16} - \epsilon)+$, one of two situations happen, either the players continue to receive the information that the payment is given by $f_0$, and then the actual initial distribution of players was $\delta_0$, or they see a change of information and the initial distribution of players was given by $\delta_\epsilon$. In any case, the situation then reduces to a standard MFG (which admits a unique solution thanks to the monotonicity of the coupling here).

Of course, this example is quite simple, in particular, the control of the players do not depend much on the belief. However, we believe it is instructive to keep in mind. Furthermore, we insist on the fact that there is some room in this example to consider more general situations. In particular, we could make small perturbations of $f$ by functions depending more generally on the restriction of $m$ to $[\frac{5}{16},\frac{3}{8}]$.

\subsection{Derivation of the master equation}
The associated master equation is derived formally in this section. We start by a development on functions on $[0,T]\times \mpptd$ of the form  
$$
V(t,µ) := \mathbb{E}_{µ}\left[\int_t^T G(\nu^{m,µ}_s)ds\right],
$$
for smooth $G$ where $(\nu^{m,µ}_t)_{t \geq 0}$ is the process constructed in the proof of Proposition \ref{proclearn} when $\bar{µ} \leftarrow µ$ and where $\mathbb{E}_µ[G(m)] := \int_{\mptd}G(m)µ(dm).$
For any $dt > 0$, separating the previous integral leads to
\be\label{ddp}
V(t,µ) = \mathbb{E}_{µ}\left[\int_t^{t+dt} G(\nu^{m,µ}_s)ds + V(t+dt,\nu^{m,µ}_{t+dt})\right].
\ee
 If on one hand the evolution of $(\nu^{m,µ}_s)_{s \geq t}$ has no reason to be smooth, on the other hand, because the information is precisely through $f$, for $µ$ almost every $m$, $(f(\nu^{m,µ}_s))_{s \geq t}$ should be smooth. Hence if $G$ is given as 
 $$
 G(µ) := \int_{\mptd}\Psi(f(m))µ(dm)
 $$
 for a smooth function $\Psi$, then we expect that $G$ is sufficiently smooth so that for $µ$ almost every $m$
 $$
(dt)^{-1}\int_t^{t+dt} G(\nu^{m,µ}_s)ds \underset{dt \to 0}{\longrightarrow} G(µ).
 $$
 Thus, dividing by $dt$ and letting $dt \to 0$ in \eqref{ddp} leads, formally, to the PDE
 $$
 -\partial_t V - A[µ,b,f][V] = G(µ),
 $$
 where $A[µ,b,f][V]$ is the operator "defined" by
 \be\label{defA}
 A[µ,b,f][V] := \lim_{dt \to 0} \mathbb{E}_{µ}\left[ \frac{V(\nu^{m,µ}_{dt}) - V(µ)}{dt}   \right],
 \ee
where $(\nu^{m,µ}_t)_{t \geq 0}$ is the process given by Proposition \ref{proclearn}, when the starting belief is $µ$ and $m$ the initial, unknown, distribution of players.\\

Clearly, because the evolution of $(\nu^{m,µ}_t)_{t \geq 0}$ is neither smooth nor even well defined, the domain of definition of $A$ is not clear at all. We shall come back on this question later on.\\

We are now equipped to derive the master equation. Fixing the strategies of the "other" players through the function $b : [0,T]\times \mathbb{T}^d\to \mathbb{R}^d$, the value function of a player is
$$
\begin{aligned}
U&(t,x,µ) =\\
& \inf_{\alpha} \mathbb{E}_{\mathbb{P},µ}\left[\int_t^T\int_{\mptd}f(m')(X^{\alpha}_s)\nu^{m,µ}_s(dm') + L(X^{\alpha}_s,\alpha_s)ds + \int_{\mptd}U_0(m')(X^{\alpha}_T)\nu^{m,µ}_T(dm')\right],
\end{aligned}
$$
where the state $(X^{\alpha}_s)_{s \geq 0}$ evolves according to \eqref{sde} and the infimum is taken over all progressively measurable process (with respect to both the individual noise $(W_s)_{s \geq 0}$ and the common information). Hence, following the previous development, formally, if $U$ is smooth it is a solution of
$$
\begin{aligned}
-\partial_t U - \sigma \Delta U + H(x,\nabla_x U) - A[µ,b,f][U] = \tilde{f}(µ)(x) \text{ in } (0,T)\times \T^d\times \mathcal{A},\\
U(T,x,µ) = \tilde{U}_0(µ)(x) \text{ in } \T^d\times \mathcal{A}.
\end{aligned}
$$
Where $\tilde{f}$ and $\tilde{U}_0$ are defined as in the previous section. 

Replacing $b$ by what should be the optimal strategies of the players, and reversing time to lighten notations, one obtains the master equation 
\be\label{master:a}
\begin{aligned}
\partial_t U - \sigma \Delta U + H(x,\nabla_x U) - A[µ,-D_pH(\nabla_x U),f][U] = \tilde{f}(µ)(x) \text{ in } \R_+\times \T^d\times \mathcal{A},\\
U(0,x,µ) = \tilde{U_0}(µ)(x) \text{ in } \T^d\times \mathcal{A}.
\end{aligned}
\ee

\subsection{Mathematical analysis of the master equation}
This section contains a partial mathematical analysis of the master equation just derived. The notion of monotone solutions introduced in \citep{bertucci2021monotone,bertucci2021monotone2} is used to prove some properties of the value functions for such MFG.\\

Even if the precise nature of the operator $A$ is not established here, it possesses the following properties.
\begin{enumerate}
\item If there is no learning (e.g. if $f$ is constant), then $\nu^{m,µ}_t = (K_t)_{\#}µ$ for any $m$ and $A$ is defined on smooth functions on $\mathcal{A} = \mpptd$.
\item When evaluated on the minimum $µ^*$ of a function $V : \mpptd \to \mathbb{R}$, one should have $A[µ^*,b,f][V] \geq 0$ for any $b,f$.
\item If, for some smooth function $\phi$, $V$ is a function of the form
\be\label{funclinear}
V(µ) = \int_{\mptd}\phi(m)µ(dm),
\ee
then, whatever the function $f$,
 \be\label{propAlinear}
 A[µ,b,f][V] = \int_{\mptd} \int_{\T^d} \sigma \Delta_x\nabla_m\phi(m,x) + b (t,x)\cdot \nabla_x\nabla_m\phi(m,x)m(dx) µ(dm).
 \ee
\end{enumerate}
The first point is mainly a remark, the second one follows from \eqref{defA} and the third point is a consequence of \eqref{star3}. Indeed, it is a consequence of the fact that because $V$ is linear, $\mathbb{E}_{µ}[V(\nu^{m,µ}_{dt})] = V((K_t)_{\#}µ)$.

Although the description of the operator $A$ is quite poor at this time, the properties (2) and (3) above are sufficient to define a concept of monotone solutions here. 

\subsubsection{Monotone solutions of master equations} The main advantage of monotone solutions is that they allow to define solutions of \eqref{master:a} without using the operator $A$ directly on $U$ but on a large set of simpler functions instead, namely functions of the form \eqref{funclinear}. To define precisely these simpler functions, and to make the following more understandable, one needs to use a duality between $\mathcal{C}(\mathbb{T}^d)$ and $\mpptd$. As suggested by the computation done in the blind case, we choose the following duality between $\phi \in \mathcal{C}(\mathbb{T}^d)$ and $µ \in \mathcal{M}(\mptd)$:
$$
\langle \phi, µ \rangle := \int_{\mptd}\int_{\mathbb{T}^d}\phi(x)m(dx)µ(dm).
$$

The idea of monotone solutions consists in looking at minima of the function 
$$
W (t,µ) := \langle U(t,\cdot,µ) - \phi(\cdot),µ - \nu\rangle
$$ for $\phi \in \mathcal{C}^2(\mathbb{T}^d)$ and $\nu \in \mathcal{M}(\mptd)$. Formally, $W$ is a solution of
\be\label{master:a}
\begin{aligned}
\partial_t W - \sigma\langle \Delta U   + H(x,\nabla_x U) ,µ - \nu\rangle - A[µ,-D_pH(\nabla_x U),f][W]\\
 = \langle \tilde{f},µ-\nu\rangle - \left\langle \sigma\Delta (U - \phi) - D_pH(\nabla_xU)\cdot\nabla_x(U - \phi),µ\right\rangle.
\end{aligned}
\ee
Indeed, remark that 
$$
A[µ,-D_pH(\nabla_x U),f][W] = \langle A[µ,-D_pH(\nabla_x U),f][U], µ - \nu \rangle + A[µ,-D_pH(\nabla_x U),f][\Psi]
$$
where $\Psi : µ' \to \langle U(µ) - \phi,µ'\rangle$. The previous formula is just the equivalent of the formula $(fg)' = f'g + g'f$ for the operator $A$. Now using \eqref{propAlinear} we deduce that 
\be\label{eq43}
A[µ,-D_pH(\nabla_x U),f][\Psi] = \left\langle \sigma\Delta (U - \phi) - D_pH(\nabla_xU)\cdot\nabla_x(U - \phi),µ\right\rangle.
\ee

Using both the facts that on minima of $W$, $A[µ,-D_pH(\nabla_x U),f][W] \geq 0$, and \eqref{eq43}, we arrive at the

\begin{Def}\label{def:mon}
We say that a continuous function $U : [0,T]\times \mathbb{T}^d \times \mathcal{A}$, smooth in its second argument, is a value of the MFG with observed payments and unknown distribution of players if : 
 \begin{itemize} \item for any $\mathcal{C}^2$ function $\phi : \mathbb{T}^d \to \mathbb{R}$, for any measure $\nu \in \mathcal{M}(\mathcal{M}(\mathbb{T}^d))$, for any smooth function $\vartheta : [0,\infty) \to \mathbb{R}$ and any point $(t_0,µ_0) \in (0,T]\times\mathcal{A}$ of minimum of $(t,µ) \to \langle U(t,\cdot,µ)- \phi, µ- \nu \rangle - \vartheta(t)$ on $(0,t_0]\times\mathcal{A}$, the following holds
$$
\begin{aligned}
\frac{d \vartheta}{dt}(t_0)& +  \langle -\sigma \Delta U + H(\cdot,\nabla_x U), µ_0 - \nu \rangle \geq \langle \tilde{f}(\cdot,µ_0),µ_0 - \nu\rangle\\
&-  \left\langle \sigma\Delta (U - \phi) - D_pH(\nabla_xU)\cdot\nabla_x(U - \phi),µ_0\right\rangle.
\end{aligned}
$$
\item the initial condition holds
$$
U(0,x,µ) = \tilde{U}_0(x,µ) \text{ in } \T^d\times \mathcal{A}.
$$
\end{itemize}
\end{Def}

 \begin{Rem}
The previous definition only involves the information of the payments through the set $\mathcal{A}$, on which the value function is defined. Defining it on a larger set than $\mathcal{A}$ would be meaningless since such belief are not coherent with the model.
\end{Rem}
\subsubsection{A first result toward uniqueness of monotone solutions}
Ideally, following Remark \ref{rem:payments}, one could hope to establish a uniqueness result for value functions in the sense of Definition \ref{def:mon}. However because of the nature of the set $\mathcal{A}$, we have not been able to prove such a result in a general framework. The nature of the set $\mathcal{A}$ can be described with the following result.

\begin{Prop}
Assume $f: \mptd \mapsto \mathcal C(\T^d)$ is continuous. Then, the set $\mathcal{A}$ is compact (for the weak topology). As soon as $f$ is neither one-to-one nor constant, $\mathcal{A}$ is not convex.
\end{Prop}
\begin{proof}
Since $\mpptd$ is compact for the weak topology, if $\mathcal A$ is closed, then it is compact. Remark that $µ \in \mathcal A$ if and only if
$$
\int_{\mptd\times\mptd}\|f(m_1)-f(m_2)\|_\inftyµ(dm_1)\otimesµ(dm_2) = 0.
$$
Because the integrand in the previous expression is a continuous function of $m_1$ and $m_2$, we deduce that the right hand side of the previous expression is continuous in $µ$ for the weak topology. Hence, $\mathcal A$ is closed as the pre-image of $\{0\}$ by this map.\\

If $f$ is neither one-to-one or constant, take $m_1,m_2$ and $m_3$ such that $f(m_1) = f(m_2) \ne f(m_3)$ and remark that $\frac12 \delta_{m_1} + \frac12 \delta_{m_2}$ and $\delta_{m_3}$ belongs to $\mathcal{A}$ while none of their (strict) convex combination does.
\end{proof}
\begin{Rem}
Rigorously, $\mathcal{A}$ is not convex even if $f$ is one-to-one. However in this case, it is isomorphic to $\mptd$ which is convex. I is not clear that this isomorphism preserves convexity, however, this remark is still helpful and I shall come back on this later on.
\end{Rem}

Even if we were not able to establish a general result of uniqueness, we could prove the following, which is commented immediately afterwards.

 \begin{Theorem}\label{thm:uniq}
Assume $f:\mptd \mapsto \mathcal C(\T^d)$ is continuous. If $f$ and $U_0$ are monotone, then two value functions $U$ and $V$ of the MFG, in the sense of Definition \ref{def:mon}, are such that for any $t\geq 0, µ \in \mathcal{A}, m,m' \in \mptd$ such that $f(m) = f(m') = \tilde{f}(µ)$
 $$
 \int_{\mathbb{T}^d}U(t,x,µ) - V(t,x,µ)(m-m')(dx) = 0.
 $$
Moreover, for any $t \geq 0, m \in \mptd$, $\nabla_x U(t,\cdot,\delta_m) = \nabla_xV(t,\cdot,\delta_m)$. 
 \end{Theorem}
 \begin{Rem}
 The last part of the Theorem only states that when the players know the distribution of players, only one profile of strategies is possible for the players, which is the one in the case with full information. The equality can be improved to equality of the value on Dirac masses following \citep{bertucci2021monotone2} under stronger assumptions on the monotonicity of $f$. The first part of the result is a bit less classical. It allows to characterize a set to which $U-V$ is orthogonal, at any point $t,µ$. Remark that the larger is $f^{-1}(\{ \tilde{f}(µ)\})$, the more information we have on the difference. However this set is in general not sufficiently large to deduce, in a simple way, uniqueness results for value functions of the MFG.
 \end{Rem}
 \begin{proof}
Assume that there exists $t \geq 0, µ, \nu \in \mathcal{A}$ such that 
$$
\langle U(t,µ) - V(t,\nu), µ - \nu \rangle = - c_0 < 0.
$$
Then, there exists $\delta > 0$ such that for any $\alpha$, the function $W$ defined by 
 $$
 W(t,s,µ,\nu) = \langle U(t,µ) - V(s,\nu), µ- \nu \rangle + \alpha(t-s)^2 + \delta(t + s)
 $$
 is not non-negative, with a minimum lower than $-\frac12 c_0$. Since $\mathcal{A}$ is compact and $W$ continuous, consider a minimum of $W$ on $[0,T]^2\times \mathcal{A}^2$, denoted by $(t_*,s_*,µ_*,\nu_*)$.
 
 Assume first that $t_*,s_* > 0$. Using the fact that $U$ is a value function of the MFG, we obtain 
 $$
 \begin{aligned}
 -&\delta - 2 \alpha(t-s) +  \langle -\sigma \Delta U(t_*,µ_*) + H(\cdot,\nabla_x U), µ_* - \nu_* \rangle \geq \langle \tilde{f}(\cdot,µ_*),µ_* - \nu_*\rangle\\
&-  \left\langle \sigma\Delta (U(t_*,µ_*) - V(s_*,\nu_*)) - D_pH(\nabla_xU)\cdot\nabla_x(U - V),µ_*\right\rangle.
\end{aligned}
 $$
 The analogous relation for $V$ is
  $$
 \begin{aligned}
 -&\delta - 2 \alpha(s-t) +  \langle -\sigma \Delta V(s_*,\nu_*) + H(\cdot,\nabla_x V), \nu_* - \mu_* \rangle \geq \langle \tilde{f}(\cdot,\nu_*),\nu_* - \mu_*\rangle\\
&-  \left\langle \sigma\Delta (V(s_*,\nu_*) - U(t_*,\mu_*)) - D_pH(\nabla_xV)\cdot\nabla_x(V - U),\nu_*\right\rangle.
\end{aligned}
 $$
Summing the two previous inequality, cancelling the terms in $\sigma$ and $\alpha$ yields
 $$
 \begin{aligned}
 -2 \delta &+  \langle H(\cdot,\nabla_x U), µ_* - \nu_* \rangle + \langle  H(\cdot,\nabla_x V), \nu_* - µ_* \rangle \geq\\
 & \langle \tilde{f}(\cdot,µ_*) - \tilde{f}(\cdot,\nu_*),µ_* - \nu_*\rangle -  \left\langle D_pH(\nabla_xU)\cdot\nabla_x(U - V),µ_*\right\rangle\\
 & - \langle D_pH(\nabla_xV)\cdot\nabla_x(V - U),\nu_*\rangle.
 \end{aligned}
 $$
From the convexity of the Hamiltonian, we deduce that
$$
\langle H(\cdot,\nabla_x V) - H(\cdot,\nabla_x U) - D_pH(\cdot,\nabla_x U)\cdot\nabla_x(V-U),µ_*\rangle \geq 0.
$$
Using the analogous relation by exchanging $U$ and $V$ and replacing $µ_*$ by $\nu_*$, we obtain
$$
-2 \delta \geq  \langle \tilde{f}(\cdot,µ_*) - \tilde{f}(\cdot,\nu_*),µ_* - \nu_*\rangle.
$$
Using the monotonicity of $f$, we finally arrive at 
$$
-2\delta \geq 0
$$ which is a contradiction. Hence, we have that either $t_* = 0$ or $s_*= 0$. Without loss of generality, assume that $t_* = 0$. Let us remark that we necessary have that $\alpha(t_* - s_*)^2$ is bounded, hence, taking $\alpha$ as big as we want, we obtain that $s_*$ can be as close as $0$ as we like. Consider now $\alpha$ sufficiently big so that $\|V(s_*,\cdot,\cdot) - U_0(\cdot,\cdot)\|_{\infty} \leq \frac{c_0}{3}$. Evaluating $W$ at its minimum yields
$$
\langle U_0(\cdot,µ_*) - U_0(\cdot,\nu_*),µ_*- \nu_*\rangle + \alpha (s_*)^2 + \delta s_* + \langle U_0(\cdot,\nu_*) - V(s_*,\cdot,\nu_*),µ_*-\nu_*\rangle \leq -\frac{c_0}{2}.
$$
Since $U_0$ is monotone, we deduce that
$$
-\frac{c_0}{3} \leq -\frac{c_0}{2}.
$$
Thus, we also obtain a contradiction in this case.\\
 
 Hence we deduce that for every $µ,\nu \in \mathcal{A}, t \geq 0$, 
 \be\label{wpositive}
 \langle U(t,µ)-V(t,\nu), µ-\nu \rangle \geq 0.
 \ee
 We now explain how this information translates into the required result. Consider $\bar{µ} \in \mathcal{A}$, $\theta \in(0,1)$ and two measures $m_1,m_2 \in \mptd$ such that $f(m_1) = f(m_2) = \tilde{f}(\bar{µ})$. Using \eqref{wpositive} for $µ = (1- \theta)\bar{µ} + \theta \delta_{m_1}$ and $\nu = (1-\theta)\bar{µ} + \theta \delta_{m_2}$, we obtain that 
 $$
\theta \langle U(t,µ) - V(t,\nu),\delta_{m_1} - \delta_{m_2}\rangle \geq 0.
 $$
 Dividing by $\theta$ and letting $\theta \to 0$ yields 
 $$
 \int_{\mathbb{T}^d}U(t,x,\bar{µ}) - V(t,x,\bar{µ}) (m_1 - m_2)(dx) \geq 0.
 $$
 By symmetry of $m_1$ and $m_2$, we deduce that
  $$
 \int_{\mathbb{T}^d}U(t,x,\bar{µ}) - V(t,x,\bar{µ}) (m_1 - m_2)(dx) = 0.
 $$
 Consider now $m \in \mptd$. Taking $m_1,m_2 \in \mptd$, $\theta \in (0,1)$, and defining $\bar{µ} = \delta_m$, $µ = \delta_{(1- \theta)m + \theta m_1}$ and $\nu = \delta_{(1-\theta)m + \theta m_2}$, we obtain by proceeding as immediately above, 
$$
 \int_{\mathbb{T}^d}U(t,x,\bar{µ}) - V(t,x,\bar{µ}) (m_1 - m_2)(dx) = 0,
 $$
 which proves the second part of the claim.
 
 \end{proof}
 \begin{Rem}
 We hope the computation at the end of the previous result hint at why the set $\mathcal A$ can be thought of as convex in the case in which $f$ is one-to-one.
 \end{Rem}
 
 In the case of a dependence through the first moment only, the previous in fact yields uniqueness as the next result shows.
 
 \begin{Prop}
 Under the assumptions of the previous result, if $f$ only depends on $m$ through its first moment $E(m) := \int_{\T^d} x m(dx)$, then for any two monotone solutions $U$ and $V$, there exists $a : \R_+\times \mathcal{A} \to \R$ such that for any $t,x,µ$ 
 $$
 U(t,x,µ) = V(t,x,µ) + a(t,µ).
 $$
 \end{Prop}
 Note that in particular, two such value functions induce the same strategies for the players. Hence, this result can be seen as a uniqueness result on the underlying equilibria of the game. Indeed, equilibrium strategies can be computed through the solution $U$ of the master equation and only depend on $\nabla_x U$. 
 \begin{proof}
 Recalling the proof of Theorem \ref{thm:uniq}, we know that for all $µ,\nu \in \mathcal{A}, t \in \R_+$,
 $$
 \langle U(t,µ) - V(t,\nu), µ - \nu \rangle \geq 0.
 $$
 Now take $µ \in \mathcal{A}$, consider $K \in \R^d$ such that almost everywhere for $µ$, $E(m) = K$. Consider any $m' \in \mptd_K:= \{m \in \mptd| E(m) = K\}$. From the previous inequality, we obtain that for any $\theta \in (0,1)$
 $$
  \theta\langle U(t,µ) - V(t,(1-\theta)µ + \theta \delta_{m'}),  µ - \delta_{m'} \rangle \geq 0.
 $$
 Hence, dividing by $\theta$ and taking the limit $\theta \to 0$, we deduce thanks to the continuity of $V$ that
 $$
 \langle U(t,µ) - V(t,µ),  µ - \delta_{m'} \rangle \geq 0.
 $$
 By symmetry, we in fact have equality in the previous relation. In particular, we obtain that there exists $a(t,µ)$ such that for any $m \in \mptd_K$,
 \be\label{eq:hj}
 \langle U(t,µ) - V(t,µ), m\rangle = a(t,µ).
 \ee
 Up to a subtracting $a$ to $U$, we argue for the moment as if $a \equiv 0$. We now argue that $h =U(t,µ)-V(t,µ)$. Indeed, observe for instance that $h(K) = 0$ by evaluating \eqref{eq:hj} on $m = \delta_K$. Now, for $x,y$ such that $\frac{x+y}{2} = K$, evaluating \eqref{eq:hj} on $\frac12 \delta_x + \frac12 \delta_y$, we deduce that $h(x) = -h(y)$. Arguing similarly $\frac13(\delta_x +\delta_y + \delta_{K -(x+y)})$, and we then obtain
 $$
 h(K + x+y) = h(K +x) + h(K+y).
 $$
 Since $h$ is continuous, this implies that $h(K+\cdot)$ is linear, but because it is an element of $\mathcal C(\T^d)$, we deduce it is in fact equal to $0$, which proves the claim.
 \end{proof}

More generally, without using the particular structure of the torus with respect to the first moment, we have the following.

 \begin{Prop}
 Let $\psi: \T^d \mapsto \R$ be a given continuous function. Under the assumptions of the previous result, if $f$ only depends on $m$ through its $\psi$-moment $E_\psi(m) := \int_{\T^d} \psi(x) m(dx)$, then for any two monotone solutions $U$ and $V$, there exists $a,b : \R_+\times \mathcal{A} \to \R$ such that for any $t,x,µ$ 
 $$
 U(t,x,µ) = V(t,x,µ) + a(t,µ) + b(t,µ)\psi(x).
 $$
 \end{Prop}
The proof is a simple non-linear adaptation of the previous one that we do not detail here.

\subsection{Comments and future perspectives}
 As already mentioned above, this study on MFG with unknown distribution of players and observed payments does not cover all questions surrounding the problem, and hopefully, more results are to come. The approach proposed here used the notion of monotone solutions of MFG master equations to obtain a definition of solutions (Definition \ref{def:mon}). If the lack of a general uniqueness result pleads against this notion of solution, this definition is nonetheless helpful to prove several properties of such value functions and a uniqueness result for a particular case. It is possible that a more restrictive notion of solution will prove to be better adapted to this problem.
 
 In the study of this problem, a fundamental question which remains open is the question of the existence of such a value function. Because the effect of the observation of the payments possesses some similarity with the presence of a common noise in MFG, some approaches to prove existence are suggested from the literature on MFG with common noise, maybe the most natural would be the one of \citep{carmona2016common}. As already mentioned above, a direct application of this approach does not seem feasible. However, if we restrict ourselves to beliefs which are combinations of Dirac masses, such a strategy looks viable. It will then suffice to have a uniform estimate on the continuity of $U$ with respect to $µ$ to pass to the limit. This approach is not presented here because we were not able to establish such a continuity estimate.
 
 \section*{Acknowledgments}
 The author is grateful to Pierre-Louis Lions and Sylvain Sorin for numerous (independent) discussions that lead to this project. The author also acknowledge a partial support from the Lagrange Mathematics and Computing Research Center and the chair FDD (Institut Louis Bachelier).
 
\bibliographystyle{plainnat}
\bibliography{bibremarks}

\end{document}